\documentclass[a4paper,english,sumlimits,reqno]{amsart}

\usepackage[T1]{fontenc}
\usepackage[utf8]{inputenc}
\usepackage{lmodern}

\usepackage{babel}

\usepackage{xifthen}
\usepackage{ifthenx}

\usepackage{xargs}

\usepackage{graphicx}
\usepackage[table,svgnames,x11names]{xcolor}

\usepackage[colorinlistoftodos,prependcaption,textsize=tiny]{todonotes}
\newcommandx{\todoin}[2][1=]{\todo[inline, caption={todo}, #1]{%
\begin{minipage}{\textwidth-20pt}#2\end{minipage}}}

\usepackage{amssymb}
\usepackage{amsmath}
\usepackage{amsthm}
\usepackage{mathtools}
\usepackage{exscale}
\usepackage{relsize}
\usepackage{bbm}
\usepackage{bm}
\usepackage{amsbsy}
\usepackage{mathdots}

\usepackage[mathcal]{eucal}
\usepackage{mathrsfs}

\usepackage{stmaryrd}

\usepackage[all]{xy}
\newcommand{\vcxymatrix}[1]{\vcenter{\xymatrix{#1}}}

\usepackage{fp}

\usepackage[section]{placeins}
\usepackage{float}

\usepackage{enumerate}

\newcounter{proof}
\newenvironment{myproof}%
{\stepcounter{proof}\begin{proof}}%
{\end{proof}}%
\newcounter{proofstep}[proof]
\newenvironment{proofstep}[1][]%
{\refstepcounter{proofstep}\bigskip\par\noindent%
  \ifthenelse{\isempty{#1}}
    {\textsc{Step \theproofstep. }}
    {\textsc{#1.}}
  \noindent}%
{\par}%
\newcounter{proofcase}[proof]
\newenvironment{proofcase}[1][]%
{\refstepcounter{proofcase}\bigskip\par\noindent%
  \ifthenelse{\isempty{#1}}
    {\textsc{Case \theproofcase. }}
    {\textsc{#1.}}
  \noindent}%
{\par}%

\usepackage{hyperref}

\theoremstyle{plain}
\newtheorem{thm}{Theorem}[section]
\newtheorem*{thm*}{Theorem}

\newtheorem{lem}[thm]{Lemma}
\newtheorem{question}[thm]{Question}
\theoremstyle{definition}

\theoremstyle{remark}

\newtheorem{rem}[thm]{Remark}

\numberwithin{equation}{section}

\newcommandx{\textref}[2][1=]{\hyperref[#2]{#1\ref*{#2}}}
\newcommandx{\textrefp}[2][1=]{(\hyperref[#2]{#1\ref*{#2}})}

\newcommand{\dif}{\ensuremath{\, \mathrm d}}

\DeclareMathOperator{\sign}{sign}
\DeclareMathOperator{\Id}{Id}

\DeclareMathOperator{\spn}{span}

\DeclareMathOperator{\diag}{diag}

\DeclareMathOperator{\cond}{\mathbb{E}}
\DeclareMathOperator{\prob}{\mathbb{P}}

\newcommand{\bmo}{\ensuremath{\mathrm{BMO}}}

\begin{document}

\title[Dimension dependence of factorization problems: Hardy spaces and $SL_n^\infty$]{Dimension
  dependence of factorization problems: Hardy spaces and $SL_n^\infty$}

\author[R.~Lechner]{Richard Lechner}

\address{Richard Lechner, Institute of Analysis, Johannes Kepler University Linz, Altenberger
  Strasse 69, A-4040 Linz, Austria}

\email{richard.lechner@jku.at}

\date{\today}



\thanks{Supported by the Austrian Science Foundation (FWF) Pr.Nr. P28352}

\begin{abstract}
  Given $1 \leq p < \infty$, let $W_n$ denote the finite-dimensional dyadic Hardy space $H_n^p$, its
  dual or $SL_n^\infty$.  We prove the following quantitative result: The identity operator on $W_n$
  factors through any operator $T : W_N\to W_N$ which has large diagonal with respect to the Haar
  system, where $N$ depends \emph{linearly} on $n$.
\end{abstract}

\maketitle

\makeatletter \providecommand\@dotsep{5} \def\listtodoname{List of Todos}
\def\listoftodos{\@starttoc{tdo}\listtodoname} \makeatother

\section{Introduction}\label{sec:intro}

\noindent
Local theory of Banach spaces is concerned with the quantitative study of finite dimensional Banach
spaces and their relation to infinite dimensional spaces and operators.  To illustrate, we give the
following example.

Suppose that for each $n\in\mathbb{N}$, the $n$-dimensional Banach space $X_n$ has a normalized
$1$-unconditional basis $e_j$, $1\leq j\leq n$, and let $e_j^*\in X_n^*$, $1\leq j \leq n$ denote
the associated coordinate functionals.
\begin{question}\label{que:factor}
  Given $n\in\mathbb{N}$ and $\delta,\Gamma,\eta > 0$, what is the smallest integer
  $N=N(n,\delta,\Gamma,\eta)$, such that for any operator $T : X_N\to X_N$ satisfying
  \begin{equation}\label{eq:que:factor:hypo}
    \|T\|\leq \Gamma
    \qquad\text{and}\qquad
    |\langle e_j^*, T e_j\rangle|
    \geq \delta,
    \quad 1\leq j\leq N,
  \end{equation}
  there are there operators $E : X_n\to X_N$ and $F : X_N\to X_n$ such that the diagram
  \begin{equation}\label{eq:que:factor:diagram}
    \vcxymatrix{X_n \ar[rr]^{\Id_{X_n}} \ar[d]_E && X_n\\
      X_N \ar[rr]_T && X_N \ar[u]_F}
    \qquad \|E\| \|F\| \leq \frac{1+\eta}{\delta}
  \end{equation}
  is commutative?
\end{question}
Note that the diagonal operator $D : X_n\to X_n$ given by $D = \delta \Id_{X_n}$, where $\Id_{X_n}$
denotes the identity operator on $X_n$, shows that for every choice for $E$ and $F$ we have
$\|E\| \|F\| \geq \frac{1}{\delta}$.

Naturally, we are interested in estimates for $N = N(n,\delta,\Gamma,\eta)$, especially in the
relation between $N$ and $n$.  For many Banach spaces, we have quantitative estimates for $N$ (see
e.g.~\cite{bourgain:1983,bourgain:tzafriri:1987,mueller:1988,blower:1990,wark:2007:class,mueller:2012,lechner:mueller:2014,lechner:2016-factor-mixed,lechner:2017-local-factor-SL}).
One would hope to obtain \emph{linear estimates} for $N$ in $n$, which, for example, has been
achieved by J.~Bourgain and L.~Tzafriri in~\cite{bourgain:tzafriri:1987} for $X_n = \ell_n^p$,
$1\leq p \leq \infty$.  However, for many other Banach spaces, the best known estimates are often
\emph{super-exponential}.

For instance, P.~F.~X.~Müller showed that for $X_{d_n} = H_n^1$, $X_{d_n} = (H_n^1)^*$
(see~\cite{mueller:1988}) and $X_{d_n} = L^p$, $1 < p < \infty$ (see~\cite{mueller:2012}), where
$d_n = {2^{n+1}-1}$, the estimate for $N$ is a \emph{nested exponential}, e.g.
\begin{equation*}
  N\leq 2^{8^n 2^{8^{n-1}2^{8^{n-2}2^{8^{n-3}2^{\iddots}}}}}.
\end{equation*}
Another example where $N$ is estimated by a nested exponential in $n$, is the one parameter space
$X_{d_n} = SL_n^\infty$ (see~\cite{lechner:2017-local-factor-SL}); a similar statement is true for
the bi-parameter mixed norm Hardy spaces $H_n^p(H_n^q)$, $1\leq p,q < \infty$ and their duals
(see~\cite{lechner:2016-factor-mixed}).

The cause for the super-exponential growth in the previous three examples can be pinpointed exactly:
the use of \emph{combinatorics}.  In this work, we introduce a new method, which replaces these
combinatorics with an \emph{entirely probabilistic} approach.  Consequently, we obtain for
$X_{d_n} = H_n^p$, $X_{d_n} = (H_n^p)^*$, $1\leq p < \infty$ and $X_{d_n} = SL_n^\infty$ (see
Theorem~\ref{thm:results:factor-1d}) the estimate
\begin{equation*}
  N \leq c n,
  \qquad\text{where $c=c(\delta,\Gamma,\eta) > 0$}.
\end{equation*}

\section{Notation}\label{sec:notation}

\noindent
The collection of \emph{dyadic intervals} $\mathcal{D}$ contained in the unit interval~$[0,1)$ is
given by
\begin{equation*}
  \mathcal{D} = \{[(k-1)2^{-n},k2^{-n}) : n\in \mathbb{N}_0, 1\leq k\leq 2^n\}.
\end{equation*}
Let $|\cdot|$ denote the Lebesgue measure.  For any $N\in\mathbb{N}_0$ we put
\begin{equation}\label{eq:dyadic-intervals}
  \mathcal{D}_N = \{I\in\mathcal{D} : |I|=2^{-N}\}
  \qquad\text{and}\qquad
  \mathcal{D}_{\leq N} = \bigcup_{n=0}^N\mathcal{D}_n.
\end{equation}
Given $n\in\mathbb{N}_0$ and a dyadic interval $I\in\mathcal{D}_n$, we define
$I^-, I^+\in\mathcal{D}_{n+1}$ by
\begin{equation}\label{eq:dyadic-intervals:children}
  I^+ \cup I^- = I
  \qquad\text{and}\qquad
  \inf I^+ < \inf I^-.
\end{equation}

The $L^\infty$-normalized \emph{Haar system} $h_I$, $I\in\mathcal{D}$ is given by
\begin{equation}\label{eq:haar-system}
  h_I = \chi_{I^+}-\chi_{I^-},
  \qquad I\in\mathcal{D},
\end{equation}
where $\chi_A$ denotes the characteristic function of a set $A\subset[0,1)$.

Given $1\leq p < \infty$, the \emph{Hardy space} $H^p$ is the completion of
\begin{equation*}
  \spn\{ h_I : I \in \mathcal{D} \}
\end{equation*}
under the square function norm
\begin{equation}\label{eq:Hp-norm}
  \Big\| \sum_{I\in \mathcal{D}} a_I h_I \Big\|_{H^p}
  = \biggl( \int_0^1 \Big(
  \sum_{I\in \mathcal{D}} a_I^2 h_I^2(x)
  \Big)^{p/2}
  \dif x
  \biggr)^{1/p}
  .
\end{equation}
For each $n\in\mathbb{N}_0$, we define the finite dimensional space
\begin{equation}\label{eq:finite-dimensional-spaces:Hp}
  H_n^p = \spn\{h_I : I\in\mathcal{D}_{\leq n}\}\subset H^p.
\end{equation}

The \emph{non-separable} Banach space $SL^\infty$ is given by
\begin{equation}\label{eq:sl-infty:space}
  SL^\infty
  = \{f\in L^2 : \|f\|_{SL^\infty} < \infty\},
\end{equation}
equipped with the norm
\begin{equation}\label{eq:sl-infty:norm}
  \Big\| \sum_{I\in\mathcal{D}} a_I h_I \Big\|_{SL^\infty}
  = \Big\| \Bigl(\sum_{I\in\mathcal{D}} a_I^2 h_I^2\Bigr)^{1/2} \Big\|_{L^\infty}.
\end{equation}
For all $n\in\mathbb{N}_0$, we define the finite dimensional space
\begin{equation}\label{eq:finite-dimensional-spaces:SL}
  SL_n^\infty
  = \spn\{h_I : I\in\mathcal{D}_{\leq n}\}\subset SL^\infty.
\end{equation}

We define the \emph{duality pairing}
$\langle \cdot, \cdot \rangle : SL^\infty\times H^1\to \mathbb{R}$ by
\begin{equation}\label{eq:scalar-product}
  \langle f, g\rangle = \int_0^1 f(x)g(x)\dif x,
  \qquad f\in SL^\infty,\ g\in H^1.
\end{equation}
An elementary computation (see e.g.~\cite{garsia:1973}) shows that
\begin{equation}\label{eq:bracket-estimate}
  |\langle f, g\rangle| \leq \|f\|_{SL^\infty} \|g\|_{H^1},
  \qquad f\in SL^\infty,\ g\in H^1.
\end{equation}

\section{Main result}\label{sec:results}

\noindent
Let $1\leq p < \infty$ and recall that we put $d_n = 2^{n+1}-1$, $n\in\mathbb{N}$.  Our main result
Theorem~\ref{thm:results:factor-1d} gives a quantitative estimate for $N = N(n,\delta,\Gamma,\eta)$ in
Question~\ref{que:factor} for the spaces $W_{d_n} = H_n^p$, $W_{d_n} = (H_n^p)^*$ and
$W_{d_n} = SL_n^\infty$.
\begin{thm}\label{thm:results:factor-1d}
  Let $1 \leq p < \infty$, and let $(W_k : k\in\mathbb{N})$ denote one of the following three
  sequences of spaces:
  \begin{align}\label{eq:thm:results:factor-1d:spaces}
    &(H_k^p : k\in\mathbb{N}), &&((H_k^p)^* : k\in\mathbb{N}),
    &&(SL_k^\infty : k\in\mathbb{N}).
  \end{align}
  Let $n \in \mathbb N$ and $\delta > 0,\Gamma,\eta > 0$.  Define the integer
  $N=N(n,\delta,\Gamma,\eta)$ by the formula
  \begin{equation}\label{eq:thm:results:factor-1d:dim}
    N
    = 19(n+2) + \bigl\lfloor
    4\log_2(\Gamma/\delta) + 4\log_2\bigl(1+\eta^{-1}\bigr)
    \bigr\rfloor
    .
  \end{equation}
  Then for any operator $T : W_N\rightarrow W_N$ satisfying
  \begin{equation}\label{eq:thm:results:factor-1d:large}
    \|T\|\leq \Gamma
    \qquad\text{and}\qquad
    |\langle T h_K, h_K \rangle|
    \geq \delta |K|,
    \quad K\in \mathcal{D}_{\leq N},
  \end{equation}
  there exist bounded linear operators $E : W_n\to W_N$ and $F : W_N\to W_n$, such that the diagram
  \begin{equation}\label{eq:thm:results:factor-1d:diag}
    \vcxymatrix{W_n \ar[rr]^{\Id_{W_n}} \ar[d]_E && W_n\\
      W_N \ar[rr]_T && W_N \ar[u]_F}
    \qquad \|E\|\|F\| \leq \frac{1+\eta}{\delta}
  \end{equation}
  is commutative.
\end{thm}
Firstly, we remark that the linear dependence of $N$ on $n$ amounts to a polynomial dependence of
the dimensions of the respective spaces; i.e. $\dim W_N$ is a polynomial in $\dim W_n$.

Secondly, although very similar in spirit, since the results of~\cite{bourgain:tzafriri:1987}
concern operators with large diagonal with respect to the standard unit vector basis in $\ell_n^p$,
the results in~\cite{bourgain:tzafriri:1987} are not applicable in the context of
Theorem~\ref{thm:results:factor-1d}, which is concerned with operators having large diagonal with respect
to the Haar system.

Thirdly, the novelty of Theorem~\ref{thm:results:factor-1d} is the above
formula~\eqref{eq:thm:results:factor-1d:dim} for $N$, specifically the \emph{linear} relation between $N$
and $n$.  Indeed, we point out that for the previous results 
\begin{itemize}
\item $(W_k : k\in\mathbb{N}) = (H_k^1 : k\in\mathbb{N})$ and
  $(W_k : k\in\mathbb{N}) = ((H_k^1)^* : k\in\mathbb{N})$ in~\cite{mueller:1988},
\item $(W_k : k\in\mathbb{N}) = (H_k^p : k\in\mathbb{N})$, $1 < p < \infty$ in~\cite{mueller:2012},
\item $(W_k : k\in\mathbb{N}) = (SL_k^\infty : k\in\mathbb{N})$ in~\cite{lechner:2017-local-factor-SL},
\end{itemize}
the relation between $N$ and $n$ is \emph{super-exponential}.  The cause for this growth is the use
of \emph{combinatorics}.  In a first step, these combinatorial methods are used to almost
diagonalize the operator $T$, and then, in a second step, probabilistic arguments are employed to
preserve the large diagonal of $T$.

By contrast, our new and \emph{entirely probabilistic} approach almost diagonalizes $T$ and
preserves its large diagonal in a \emph{single step} (see Section~\ref{sec:random}).

\section{Random block bases}\label{sec:random}

\noindent
Given $1 \leq p < \infty$, let $W_N$ denote either $W_N = H^p$, $W_N = (H^p)^*$ or
$W_N = SL_N^\infty$.  In this section, we will show that every operator $T : W_N\to W_N$ is almost
diagonalized by random block bases $\theta\mapsto b_I^{(\theta)}\subset W_N$,
$I\in\mathcal{D}_{\leq n}$.

To this end, let $\prob$ denote the uniform measure on $\{\pm 1\}^{\mathcal{D}}$, and let
$\cond$ denote the expectation with respect to the probability measures $\prob$.  Given
$n,N\in\mathbb{N}$ and pairwise disjoint sets $\mathcal{B}_I\subset\mathcal{D}_{\leq N}$,
$I\in\mathcal{D}_{\leq n}$, we define the \emph{random block basis}
\begin{equation}\label{eq:b-1d:dfn}
  b_I^{(\theta)}
  = \sum_{K\in\mathcal{B}_I} \theta_K h_K,
  \qquad I\in\mathcal{D}_{\leq n},\ \theta\in\{\pm 1\}.
\end{equation}
Given a linear operator $T : W_N\to W_N$, we define the random variables $Y_{I,I'}, Z_I$ by putting
\begin{subequations}\label{eq:dfn:rv-1d}
  \begin{align}
    Y_{I,I'}(\theta)
    & = \langle T b_I^{(\theta)},
      b_{I'}^{(\theta)}\rangle,
    & I, I'\in\mathcal{D}_{\leq n},\ I\neq I',\ \theta\in\{\pm 1\},&
       \label{eq:dfn:rv-1d:Y}\\
    Z_I(\theta)
    & = \langle T b_I^{(\theta)}, b_I^{(\theta)}\rangle
      - \sum_{K\in\mathcal{B}_I} \langle Th_K, h_K\rangle,
    & I\in\mathcal{D}_{\leq n},\ \theta\in\{\pm 1\}.&
       \label{eq:dfn:rv-1d:Z}
  \end{align}
\end{subequations}
The following Theorem~\ref{thm:var-1d} asserts that the matrix-valued random variable
$\theta\mapsto (\langle T b_I^{(\theta)}, b_{I'}^{(\theta)}\rangle)_{I,I'\in\mathcal{D}_{\leq n}}$
is for the most part (depending on the collections $\mathcal{B}_I$, $I\in\mathcal{D}_{\leq n}$)
centered around the diagonal matrix
$\diag\bigl(\sum_{K\in\mathcal{B}_I} \langle Th_K, h_K\rangle\bigr)_{I\in\mathcal{D}_{\leq n}}$.
\begin{thm}\label{thm:var-1d}
  Let $n,N\in\mathbb{N}$, and let $\mathcal{B}_I\subset\mathcal{D}_{\leq N}$,
  $I\in\mathcal{D}_{\leq n}$ denote non-empty collections of dyadic intervals satisfying
  \begin{subequations}\label{eq:thm:var-1d:coll}
    \begin{align}
      \mathcal{B}_I\cap \mathcal{B}_{I'}
      &= \emptyset,
      &&I,I'\in\mathcal{D}_{\leq n}, I\neq I'
         \label{eq:thm:var-1d:coll:a}\\
      K\cap K' &= \emptyset,
      &&K, K'\in\mathcal{B}_I, K\neq K',\ I\in\mathcal{D}_{\leq n}.
         \label{eq:thm:var-1d:coll:b}
    \end{align}

  \end{subequations}

  Define $\alpha$ by
  putting
  \begin{equation}\label{eq:alpha-small-1d}
    \alpha = \max\{|K| : K\in\mathcal{B}_I,\ I\in\mathcal{D}_{\leq n}\}.
  \end{equation}
  Given $1 \leq p < \infty$, let $W_N$ denote either $W_N = H^p$, $W_N = (H^p)^*$ or
  $W_N = SL_N^\infty$.  Then for any operator $T : W_N\to W_N$, we have that
  \begin{align}\label{eq:lem:var-1d:exp}
    \cond Y_{I,I'}
    &= \cond Z_I
      = 0,
    && I,I'\in\mathcal{D},\ I\neq I',
  \end{align}
  and the random variables $Y_{I,I'},Z_I$ satisfy the estimates
  \begin{align}
    \label{eq:lem:var-1d:var}
    \cond Y_{I,I'}^2 &\leq \|T\|^2 \alpha^{1/2},
    &\cond Z_I^2 &\leq 2 \|T\|^2 \alpha^{1/2},
  \end{align}
  for all $I,I'\in\mathcal{D},\ I\neq I'$.
\end{thm}

Before we proceed to the proof of Theorem~\ref{thm:var-1d}, we record the following elementary facts.
\begin{lem}\label{lem:block:basic-estimate-1d}
  Let $\mathcal{B}$ be a non-empty, finite collection of pairwise disjoint dyadic intervals, and
  define
  \begin{equation}\label{lem:block:basic-estimate-1d:func}
    b^{(\theta)}
    = \sum_{K\in\mathcal{B}} \theta_K h_K,
    \qquad \theta \in \{\pm 1\}^{\mathcal{D}}.
  \end{equation}
  Then for all $1\leq p < \infty$, $1< p' \leq \infty$ with $\frac{1}{p}+\frac{1}{{p'}} = 1$, we
  have
  \begin{equation}
    \label{eq:lem:block:basic-estimate-1d}
    \|b^{(\theta)}\|_{H^p} = \Big|\bigcup\mathcal{B}\Big|^{1/p},
    \qquad
    \|b^{(\theta)}\|_{(H^p)^*} = \Big|\bigcup\mathcal{B}\Big|^{1/{p'}}
    \qquad\text{and}\qquad
    \|b^{(\theta)}\|_{SL^\infty} = 1.
  \end{equation}
\end{lem}

\begin{proof}[Proof of Lemma~\ref{lem:block:basic-estimate-1d}]
  The proof is simple and straightforward, and therefore omitted.
\end{proof}

\begin{proof}[Proof of Theorem~\ref{thm:var-1d}]
  Clearly, $\cond Y_{I,I'} = \cond Z_I = 0$.

  Note that for $K_0,K_1,K_0',K_1'\in\mathcal{D}$, we have
  $\cond \theta_{K_0}\theta_{K_1}\theta_{K_0'}\theta_{K_1'} = 1$ if and only if one of the following
  conditions is satisfied:
  \begin{enumerate}[(K1)]
  \item\label{enu:proof:lem:var:k1} $K_0 = K_1 = K_0' = K_1'$;
  \item\label{enu:proof:lem:var:k2} $K_0 = K_1\neq K_0' = K_1'$;
  \item\label{enu:proof:lem:var:k3} $K_0 = K_0'\neq K_1 = K_1'$;
  \item\label{enu:proof:lem:var:k4} $K_0 = K_1'\neq K_1 = K_0'$.
  \end{enumerate}

  \begin{proofcase}[Estimates for $Y_{I,I'}$, if $W_N = H_N^p$]
    Note that
    \begin{equation}\label{eq:proof:lem:var:Y-1d}
      \cond Y_{I,I'}^2(\theta)
      = \sum_{\substack{K_0,K_1\in\mathcal{B}_I\\K_0',K_1'\in\mathcal{B}_{I'}}}
      \cond \theta_{K_0}\theta_{K_1}\theta_{K_0'}\theta_{K_1'}
      \langle T h_{K_0}, h_{K_0'}\rangle
      \langle T h_{K_1}, h_{K_1'}\rangle.
    \end{equation}
    In view of~\eqref{eq:thm:var-1d:coll}
    and~\textrefp[K]{enu:proof:lem:var:k1}--\textrefp[K]{enu:proof:lem:var:k4}, the
    cases~\textrefp[K]{enu:proof:lem:var:k1}, \textrefp[K]{enu:proof:lem:var:k3},
    \textrefp[K]{enu:proof:lem:var:k4} are eliminated from the sum in~\eqref{eq:proof:lem:var:Y-1d}.
    Thus, with only~\textrefp[K]{enu:proof:lem:var:k2} terms left, \eqref{eq:proof:lem:var:Y-1d}
    reads as follows:
    \begin{equation}\label{eq:proof:lem:var:Y-1d:cond}
      \cond Y_{I,I'}^2
      = \sum_{\substack{K_0\in\mathcal{B}_I\\K_0'\in\mathcal{B}_{I'}}}
      \langle T h_{K_0}, h_{K_0'}\rangle^2.
    \end{equation}
    Put $a_{K_0,K_0'} = \langle T h_{K_0}, h_{K_0'}\rangle$ and note that
    \begin{equation}
      \label{eq:proof:lem:var:Y-1d:cond:1}
      |a_{K_0,K_0'}|
      \leq \|T\| |K_0|^{1/p} |K_0'|^{1/{p'}}.
    \end{equation}
    We will now estimate~\eqref{eq:proof:lem:var:Y-1d:cond} in two different ways.

    Firstly, we rewrite~\eqref{eq:proof:lem:var:Y-1d:cond} and then use duality to obtain
    \begin{align*}
      \cond Y_{I,I'}^2
      &= \sum_{K_0\in\mathcal{B}_I}
        \Bigl\langle T h_{K_0}, \sum_{K_0'\in\mathcal{B}_{I'}} a_{K_0,K_0'} h_{K_0'}\Bigr\rangle\\
      &\leq \sum_{K_0\in\mathcal{B}_I}
        \|T\| |K_0|^{1/p} \Bigl\|\sum_{K_0'\in\mathcal{B}_{I'}} a_{K_0,K_0'} h_{K_0'}\Bigr\|_{(H^p)^*}.
    \end{align*}
    \eqref{eq:thm:var-1d:coll}, Lemma~\ref{lem:block:basic-estimate-1d}
    and~\eqref{eq:proof:lem:var:Y-1d:cond:1} give us
    \begin{align*}
      \cond Y_{I,I'}^2
      &\leq \sum_{K_0\in\mathcal{B}_I} \|T\| |K_0|^{1/p}
        \max_{K_0'\in\mathcal{B}_{I'}} |a_{K_0,K_0'}|
        \Bigl\|\sum_{K_0'\in\mathcal{B}_{I'}} h_{K_0'}\Bigr\|\\
      & \leq \sum_{K_0\in\mathcal{B}_I} \|T\|^2 |K_0|^{2/p}
        \max_{K_0'\in\mathcal{B}_{I'}} |K_0'|^{1/{p'}}.
    \end{align*}
    Applying Hölder's inequality yields
    \begin{equation*}
      \cond Y_{I,I'}^2
      \leq \|T\|^2 \max_{\substack{K_0\in\mathcal{B}_I\\K_0'\in\mathcal{B}_{I'}}}
      |K_0|^{2/p-1}|K_0'|^{1/{p'}}.
    \end{equation*}
    Using~\eqref{eq:alpha-small-1d} gives us the estimate
    \begin{equation}\label{eq:proof:lem:var:Y-1d:est:1}
      \cond Y_{I,I'}^2 \leq \|T\|^2\alpha^{1/p}.
    \end{equation}

    Secondly, we rewrite~\eqref{eq:proof:lem:var:Y-1d:cond} as follows:
    \begin{equation*}
      \cond Y_{I,I'}^2
      = \sum_{K_0'\in\mathcal{B}_{I'}}
      \Bigl\langle T \sum_{K_0\in\mathcal{B}_{I}} a_{K_0,K_0'} h_{K_0}, h_{K_0'}\Bigr\rangle.
    \end{equation*}
    The analogous computation to the one above shows
    \begin{equation}\label{eq:proof:lem:var:Y-1d:est:2}
      \cond Y_{I,I'}^2
      \leq \|T\|^2\alpha^{1/{p'}}.
    \end{equation}

    Finally, combining~\eqref{eq:proof:lem:var:Y-1d:est:1} and~\eqref{eq:proof:lem:var:Y-1d:est:2}
    yields:
    \begin{equation}\label{eq:proof:lem:var:Y-1d:est:final}
      \cond Y_{I,I'}^2
      \leq \|T\|^2\alpha^{1/2}.
    \end{equation}
  \end{proofcase}

  \begin{proofcase}[Estimates for $Z_I$, if $W_N = H_N^p$]
    In the following sums, the variables $K_0,K_0',K_1,K_1'$ will always be summed over the
    collection $\mathcal{B}_I$.  Note that
    \begin{equation}\label{eq:proof:lem:var:Z-1d}
      \cond Z_I^2(\theta)
      = \sum_{\substack{K_0\neq K_0'\\K_1\neq K_1'}}
      \cond \theta_{K_0}\theta_{K_1}\theta_{K_0'}\theta_{K_1'}
      \langle T h_{K_0}, h_{K_0'}\rangle
      \langle T h_{K_1}, h_{K_1'}\rangle.
    \end{equation}
    In view of~\eqref{eq:thm:var-1d:coll}
    and~\textrefp[K]{enu:proof:lem:var:k1}--\textrefp[K]{enu:proof:lem:var:k4}, the
    cases~\textrefp[K]{enu:proof:lem:var:k1} and~\textrefp[K]{enu:proof:lem:var:k3}, are eliminated
    from the sum in~\eqref{eq:proof:lem:var:Z-1d}.

    If we restrict the sum in~\eqref{eq:proof:lem:var:Z-1d} to
    Case~{\textrefp[K]{enu:proof:lem:var:k2}}, \eqref{eq:proof:lem:var:Z-1d} reads
    \begin{equation}\label{eq:proof:lem:var:Z-1d:k2:cond}
      \cond Z_I^2(\theta)
      = \sum_{K_0\neq K_0'}
      \langle T h_{K_0}, h_{K_0'}\rangle^2.
    \end{equation}
    Note that the expressions~\eqref{eq:proof:lem:var:Y-1d:cond}
    and~\eqref{eq:proof:lem:var:Z-1d:k2:cond} are algebraically the same, except for the conditions
    $I\neq I'$ in~\eqref{eq:proof:lem:var:Y-1d:cond} and $I=I'$
    in~\eqref{eq:proof:lem:var:Z-1d:k2:cond}.  Hence, we can repeat the proof for $Y_{I,I'}$, which
    yields
    \begin{equation}\label{eq:proof:lem:var:Y-1d:est:k2:final}
      \cond Z_I^2
      \leq \|T\|^2\alpha^{1/2}.
    \end{equation}

    Restricting the sum in~\eqref{eq:proof:lem:var:Z-1d} to
    Case~{\textrefp[K]{enu:proof:lem:var:k4}} gives us
    \begin{equation}\label{eq:proof:lem:var:Z-1d:k4:cond}
      \cond Z_I^2(\theta)
      = \sum_{K_0\neq K_1}
      \langle T h_{K_0}, h_{K_1}\rangle
      \langle T h_{K_1}, h_{K_0}\rangle.
    \end{equation}
    Put $a_{K_0,K_1} = \langle T h_{K_0}, h_{K_1}\rangle$ and note that
    \begin{equation}
      \label{eq:proof:lem:var:Z-1d:k4:cond:1}
      |a_{K_0,K_1}|
      \leq \|T\| |K_0|^{1/p} |K_1|^{1/{p'}}.
    \end{equation}

    We will now estimate~\eqref{eq:proof:lem:var:Z-1d:k4:cond} in two different ways.  Firstly,
    rewriting~\eqref{eq:proof:lem:var:Z-1d:k4:cond} and then using duality yields
    \begin{align*}
      \cond Z_I^2
      &= \sum_{K_0} \Bigl\langle T \sum_{K_1} a_{K_0,K_1} h_{K_1}, h_{K_0}\Bigr\rangle
        \leq \sum_{K_0} \|T\| \Bigl\| \sum_{K_1} a_{K_0,K_1} h_{K_1}\Bigr\|_{H^p} |K_0|^{1/{p'}}.
    \end{align*}
    \eqref{eq:thm:var-1d:coll}, Lemma~\ref{lem:block:basic-estimate-1d}
    and~\eqref{eq:proof:lem:var:Z-1d:k4:cond:1} give us
    \begin{align*}
      \cond Z_I^2
      &\leq \sum_{K_0} \|T\| \max_{K_1} |a_{K_0,K_1}| \Bigl\| \sum_{K_1} h_{K_1}\Bigr\|_{H^p}
        |K_0|^{1/{p'}}\\
      &\leq \sum_{K_0} \|T\|^2 |K_0| \max_{K_1} |K_1|^{1/p'}
        = \|T\|^2 \max_{K_1} |K_1|^{1/p'}.
    \end{align*}
    Using~\eqref{eq:alpha-small-1d}, we obtain the estimate
    \begin{equation}\label{eq:proof:lem:var:Z-1d:k4:est:1}
      \cond Z_I^2 \leq \|T\|^2\alpha^{1/p'}.
    \end{equation}

    Secondly, we rewrite~\eqref{eq:proof:lem:var:Z-1d:k4:cond} as follows:
    \begin{equation*}
      \cond Z_I^2
      = \sum_{K_1} \Bigl\langle T h_{K_1}, \sum_{K_0} a_{K_0,K_1} h_{K_0}\Bigr\rangle.
    \end{equation*}
    The analogous computation to the one above shows
    \begin{equation}\label{eq:proof:lem:var:Z-1d:k4:est:2}
      \cond Z_I^2
      \leq \|T\|^2\alpha^{1/{p}}.
    \end{equation}

    Finally, combining~\eqref{eq:proof:lem:var:Z-1d:k4:est:1}
    with~\eqref{eq:proof:lem:var:Z-1d:k4:est:2} gives us
    \begin{equation}\label{eq:proof:lem:var:Z-1d:k4:est:final}
      \cond Z_I^2
      \leq \|T\|^2\alpha^{1/2}.
    \end{equation}
    in Case~{\textrefp[K]{enu:proof:lem:var:k4}}.

    Adding~\eqref{eq:proof:lem:var:Y-1d:est:k2:final} and~\eqref{eq:proof:lem:var:Z-1d:k4:est:final}
    yields
    \begin{equation}\label{eq:proof:lem:var:Z-1d:est:final}
      \cond Z_I^2
      \leq 2\|T\|^2\alpha^{1/2}.
      \qedhere
    \end{equation}
  \end{proofcase}

  \begin{proofcase}[Estimates for $W_N = (H_N^p)^*$ and $W_N = SL_N^\infty$]
    If $W_N = (H_N^p)^*$, we repeat the above proof, but with the roles of $H_N^p$ and $(H_N^p)^*$
    reversed.

    If $W_N = SL_N^\infty$, we only need to repeat half of the above proof (only the parts where the
    inner sum is on the $SL^\infty$ side of the duality pairing).  To be more precise, we repeat the
    proof for estimate~\eqref{eq:proof:lem:var:Y-1d:est:2} for $Y_{I,I'}$, and the proof for the
    estimates~\eqref{eq:proof:lem:var:Y-1d:est:k2:final} (which is actually repeating the proof for
    $Y_{I,I'}$, again) and~\eqref{eq:proof:lem:var:Z-1d:k4:est:1} for $Z_I$.  This way, we obtain
    the estimates
    \begin{equation}\label{eq:proof:lem:var:Z-1d:est:final:SL}
      \cond Y_{I,I'}^2
      \leq \|T\|^2\alpha
      \qquad\text{and}\qquad
      \cond Z_I^2
      \leq 2\|T\|^2\alpha.
      \qedhere
    \end{equation}
  \end{proofcase}
\end{proof}

\section{Embeddings, projections and factorization}\label{sec:factor-1d}

\noindent
First, we record essential facts about embeddings and projections in $H^p$, $(H^p)^*$,
$1 \leq p < \infty$ and $SL^\infty$, and then we prove the main result Theorem~\ref{thm:results:factor-1d}.

\subsection{Jones' compatibility condition}\label{sec:jones}

Given $\mathcal{B}_I\subset\mathcal{D}$, $I\in\mathcal{D}$, we put $B_I = \bigcup \mathcal{B}_I$.
We say that the collections $\mathcal{B}_I$, $I\in\mathcal{D}$ satisfy Jones' compatibility
condition~(\hyperref[enu:c1]{C}) (see~\cite{jones:1985}; see also~\cite{mueller:2005}) with
constant $\kappa\geq 1$, if the following four conditions are satisfied:
\begin{enumerate}[(C1)]
\item\label{enu:c1} For each $I\in\mathcal{D}$, the collection $\mathcal{B}_I$ consists of finitely
  many pairwise disjoint dyadic intervals; moreover,
  $\mathcal{B}_I\cap \mathcal{B}_{I'} = \emptyset$, whenever $I,I'\in\mathcal{D}$, $I\neq I'$.

\item\label{enu:c2} For every $I\in\mathcal{D}$, we have that $B_{I^-}\cup B_{I^+}\subset B_I$ and
  $B_{I^-}\cap B_{I^+} = \emptyset$.

\item\label{enu:c3} $\kappa^{-1} |I| \leq |B_I| \leq \kappa |I|$, for all $I\in\mathcal{D}$.

\item\label{enu:c4} For all $I_0,I\in \mathcal{D}$ with $I_0\subset I$ and $K\in \mathcal{B}_I$, we
  have $\frac{|K\cap B_{I_0}|}{|K|} \geq \kappa^{-1} \frac{|B_{I_0}|}{|B_I|}$.
\end{enumerate}

\begin{thm}\label{thm:projection-1d}
  Let $\mathcal{B}_I\subset\mathcal{D}$, $I\in\mathcal{D}$ satisfy Jones' compatibility
  condition~(\hyperref[enu:c1]{C}) with constant $\kappa = 1$.  Let
  $\theta\in \{\pm 1\}^{\mathcal{D}}$ and define
  \begin{equation}\label{eq:thm:projection-1d:tensor:1}
    b_I^{(\theta)}
    = \sum_{K\in\mathcal{B}_I} \theta_K h_K,
    \qquad I\in \mathcal{D}.
  \end{equation}
  Given $1 \leq p < \infty$, let $W$ denote either $H^p$, $(H^p)^*$ or $SL^\infty$.  Then the
  operators $B^{(\theta)}, A^{(\theta)} : W\to W$ given by
  \begin{equation}\label{eq:thm:projection-1d:operators}
    B^{(\theta)} f = \sum_{I\in \mathcal{D}} \frac{\langle f, h_I\rangle}{\|h_I\|_2^2} b_I^{(\theta)}
    \qquad\text{and}\qquad
    A^{(\theta)} f = \sum_{I\in \mathcal{D}} \frac{\langle f, b_I^{(\theta)}\rangle}{\|b_I^{(\theta)}\|_2^2} h_I
  \end{equation}
  satisfy the estimates
  \begin{equation}\label{eq:thm:projection-1d:estimates}
    \begin{aligned}
      \|B^{(\theta)} f \|_{W} & \leq \|f\|_{W},
      &&f\in W,\\
      \|A^{(\theta)} f \|_{W} &\leq \|f\|_{W}, &&f\in W.
    \end{aligned}
  \end{equation}
  Moreover, the diagram
  \begin{equation}\label{eq:thm:projection-1d:diagram}
    \vcxymatrix{W \ar[rr]^{\Id_{W}} \ar[rd]_{B^{(\theta)}} && W\\
      &  W  \ar[ru]_{A^{(\theta)}} &
    }
  \end{equation}
  is commutative and the composition $P^{(\theta)} = B^{(\theta)} A^{(\theta)}$ is the norm $1$
  projection $P^{(\theta)} : W\to W$ given by
  \begin{equation}\label{eq:thm:projection-1d:proj-formula}
    P^{(\theta)}(f)
    = \sum_{I\in\mathcal{D}} \frac{\langle f, b_I^{(\theta)}\rangle}{\|b_I\|_2^2}
    b_I^{(\theta)}.
  \end{equation}
  Consequently, the range of $B^{(\theta)}$ is complemented (by $P^{(\theta)}$), and $B^{(\theta)}$
  is an isometric isomorphism onto its range.
\end{thm}

\begin{rem}
  In~\cite{gamlen:gaudet:1973}, Gamlen and Gaudet showed a similar version of
  Theorem~\ref{thm:projection-1d} for $W=L^p$, $1 < p < \infty$.  Let us point out two major aspects of
  their method: Firstly, they are using functions $(d_i)_{i=1}^\infty$, which are not adapted to any
  dyadic filtration, therefore, their method is not applicable in $H^p$, $1 < p < \infty$.
  Secondly, condition~\textrefp[C]{enu:c4} is not part of their hypothesis.  Instead, the
  collections $\mathcal{B}_I$, $I\in\mathcal{D}$ and the sets $\{b_I^{(\theta)} = \pm 1\}$,
  $I\in\mathcal{D}$ are linked, so that their projection $P$ can be viewed as a conditional
  expectation.  Hence, $P$ is bounded in $L^1$ and their result can be extended to $L^1$.

  In~\cite[Proposition 9.6]{jmst:1979}, Johnson, Maurey, Schechtman and Tzafriri specify conditions
  for a block basis of the Haar system, so that the conclusion of Theorem~\ref{thm:projection-1d} is true
  for $W=H^p$, $1 < p < \infty$.  Since the proof relies on Stein's martingale inequality, their
  result does not extend to $W=H^1$ or $W=(H^1)^*$.  If Jones' compatibility
  condition~(\hyperref[enu:c1]{C}) is satisfied, the operator $B^{(\theta)}$ and the projection
  $P^{(\theta)}$ in Theorem~\ref{thm:projection-1d} are the same as the respective operators occurring
  in~\cite[Proposition 9.6]{jmst:1979}.

  In~\cite{jones:1985}, Jones showed Theorem~\ref{thm:projection-1d} for $W=H^1$ and $W=(H^1)^*$.  In order
  to achieve this, it was crucial to have condition~\textrefp[C]{enu:c4} in place.

  The case $W=SL^\infty$ is proved in~\cite{lechner:2016:factor-SL}, even without
  requiring~\textrefp[C]{enu:c3}.
\end{rem}

\subsection{Proof of the main result Theorem~\ref{thm:results:factor-1d}}\label{sec:factor-1d-proof}

For convenience, we repeat Theorem~\ref{thm:results:factor-1d} here.
\begin{thm}[Main result Theorem~\ref{thm:results:factor-1d}]\label{thm:factor-1d}
  Let $1 \leq p < \infty$, and let $(W_k : k\in\mathbb{N})$ denote one of the following three
  sequences of spaces:
  \begin{align}\label{eq:thm:factor-1d:spaces}
    &(H_k^p : k\in\mathbb{N}), &&((H_k^p)^* : k\in\mathbb{N}),
    &&(SL_k^\infty : k\in\mathbb{N}).
  \end{align}
  Let $n \in \mathbb N$ and $\delta,\Gamma,\eta > 0$.  Define the integer
  $N=N(n,\delta,\Gamma,\eta)$ by the formula
  \begin{equation}\label{eq:thm:factor-1d:dim}
    N
    = 19(n+2) + \bigl\lfloor
    4\log_2(\Gamma/\delta) + 4\log_2\bigl(1+\eta^{-1}\bigr)
    \bigr\rfloor
    .
  \end{equation}
  Then for any operator $T : W_N\rightarrow W_N$ satisfying
  \begin{equation}\label{eq:thm:factor-1d:large}
    \|T\|\leq \Gamma
    \qquad\text{and}\qquad
    |\langle T h_K, h_K \rangle|
    \geq \delta |K|,
    \quad K\in \mathcal{D}_{\leq N},
  \end{equation}
  there exist bounded linear operators $E : W_n\to W_N$ and $F : W_N\to W_n$, such that the diagram
  \begin{equation}\label{eq:thm:factor-1d:diag}
    \vcxymatrix{W_n \ar[rr]^{\Id_{W_n}} \ar[d]_E && W_n\\
      W_N \ar[rr]_T && W_N \ar[u]_F}
    \qquad \|E\|\|F\| \leq \frac{1+\eta}{\delta}
  \end{equation}
  is commutative.
\end{thm}

\begin{myproof}
  Define the norm $1$ multiplication operator $M : W_N\to W_N$ as the linear extension of
  \begin{equation*}
    h_K\mapsto \sign(\langle T h_K, h_K\rangle) h_K,
    \qquad K\in\mathcal{D}_{\leq N},
  \end{equation*}
  and observe that by~\eqref{eq:thm:factor-1d:large}, we obtain
  \begin{equation*}
    \langle T M h_K, h_K \rangle
    = |\langle T h_K, h_K \rangle|
    \geq \delta |K|,
    \qquad K\in\mathcal{D}_{\leq N}.
  \end{equation*}
  Thus, we can assume that
  \begin{equation}
    \label{eq:proof:thm:factor-1d:large}
    \langle T h_K, h_K \rangle\geq \delta|K|,
    \qquad K\in\mathcal{D}_{\leq N}.
  \end{equation}

  Before we proceed, we define the following two constants: Let $m_0\in\mathbb{N}_0$ be the smallest
  integer for which
  \begin{equation}\label{proof:thm:factor-1d:const:1}
    2^{m_0} > \frac{2^{6(n+2)}\Gamma^4}{\eta_0^4},
    \qquad\text{where}\quad
    \eta_0 = \frac{\eta\delta}{(1+\eta)2^{3(n+2)}}.
  \end{equation}
    
  \begin{proofstep}[Step~\theproofstep: Overview]
    The operators $E$ and $F$ will be defined in terms of a block basis $b_I^{(\theta)}$,
    $I\in\mathcal{D}_{\leq n}$ of the Haar system $h_K$, $K\in\mathcal{D}_{\leq N}$ having the
    following form:
    \begin{equation}\label{proof:thm:factor-1d:overview:0}
      b_{I}^{(\theta)}
      = \sum_{K\in\mathcal{B}_I} \theta_K h_K,
      \qquad I\in\mathcal{D}_{\leq n},\ \theta\in\{\pm 1\}^{\mathcal{D}}.
    \end{equation}
    Our goal is to find collections $\mathcal{B}_I\subset\mathcal{D}_{\leq N}$,
    $I\in\mathcal{D}_{\leq n}$ satisfying Jones' compatibility condition~(\hyperref[enu:c1]{C}) with
    constant $\kappa=1$, and signs $\theta\in \{\pm 1\}^{\mathcal{D}}$ such that
    \begin{subequations}\label{proof:thm:factor-1d:overview:1}
      \begin{align}
        |\langle T b_I^{(\theta)},
        b_{I'}^{(\theta)}\rangle|
        & \leq \eta_0,
        & I,I'\in\mathcal{D}_{\leq n},\ I\neq I',&
                                                   \label{proof:thm:factor-1d:overview:1:a}\\
        \langle T b_I^{(\theta)},
        b_I^{(\theta)}\rangle
        & \geq (\delta - 2^{n}\eta_0)\|b_I^{(\theta)}\|_2^2,
        & I\in\mathcal{D}_{\leq n}.&
                                     \label{proof:thm:factor-1d:overview:1:b}
      \end{align}
    \end{subequations}
  \end{proofstep}

  \begin{proofstep}[Step~\theproofstep: constructing the random block basis $b_I^{(\theta)}$,
    $I\in\mathcal{D}_{\leq n}$]\label{proofstep:thm:factor-1d:2}
    First, we will use a minimalist Gamlen-Gaudet construction to define the collections
    $\mathcal{B}_I$, $I\in\mathcal{D}_{\leq n}$, and then we will rely on Theorem~\ref{thm:var-1d} to find
    signs $\theta\in\{\pm 1\}^{\mathcal{D}}$ such that~\eqref{proof:thm:factor-1d:overview:1} is
    satisfied.
    
    We will now inductively define the collections $\mathcal{B}_I$, $I\in\mathcal{D}_{\leq n}$.  We
    begin by putting,
    \begin{equation}\label{eq:proof:thm:factor-1d:coll:1}
      \mathcal{B}_{[0,1)}
      = \mathcal{D}_{m_0}.
    \end{equation}
    Let $0\leq k \leq n-1$ and assume that we have already constructed the collections
    $\mathcal{B}_I$, $I\in\mathcal{D}_{\leq k}$.  Then, we define
    \begin{equation}\label{eq:proof:thm:factor-1d:coll:2}
      \mathcal{B}_{I^+}
      = \{ K^+ : K\in\mathcal{B}_I\}
      \quad\text{and}\quad
      \mathcal{B}_{I^-}
      = \{ K^- : K\in\mathcal{B}_I\},
      \qquad I\in\mathcal{D}_k.
    \end{equation}
    One can easily verify that the collections, $\mathcal{B}_I$, $I\in\mathcal{D}_{\leq n}$ satisfy
    Jones' compatibility condition~(\hyperref[enu:c1]{C}) with constant $\kappa = 1$.

    Next, we will use a probabilistic argument to find $\theta\in\{\pm 1\}^{\mathcal{D}}$ such
    that~\eqref{proof:thm:factor-1d:overview:1} is satisfied.  To this end, let us now define the
    off-diagonal events
    \begin{subequations}\label{eq:dfn:rv-1d:tail:2}
      \begin{equation}\label{eq:dfn:rv-1d:tail:2:a}
        O_{I,I'} = \bigl\{ \theta\in\{\pm 1\}^{\mathcal{D}} :
        |\langle T b_I^{(\theta)}, b_{I'}^{(\theta)}\rangle|
        > \eta_0
        \bigr\},
        \qquad I,I'\in\mathcal{D}_{\leq n},\ I\neq I'
      \end{equation}
      and the diagonal events
      \begin{equation}\label{eq:dfn:rv-1d:tail:2:b}
        D_I = \biggl\{ \theta\in\{\pm 1\}^{\mathcal{D}} :
        \Bigl|\langle T b_{I}^{(\theta)}, b_{I}^{(\theta)}\rangle
        - \sum_{K\in\mathcal{B}_I} \langle Th_{K}, h_{K}\rangle
        \Bigr|
        > \eta_0
        \biggr\},
        \qquad I\in\mathcal{D}_{\leq n}.
      \end{equation}
    \end{subequations}
    By Theorem~\ref{thm:var-1d} and the definition of the random variables $Y_{I,I'},Z_I$
    (see~\eqref{eq:dfn:rv-1d}), we obtain
    \begin{equation}\label{eq:dfn:rv-1d:tail:1}
      \prob (O_{I,I'})
      \leq \frac{\Gamma^2}{2^{m_0/2}\eta_0^2}
      \quad\text{and}\quad
      \prob (D_I)
      \leq \frac{2\Gamma^2}{2^{m_0/2}\eta_0^2},
      \qquad I,I'\in\mathcal{D}_{\leq n},\ I\neq I.
    \end{equation}
    Using~\eqref{eq:dfn:rv-1d:tail:1} and~\eqref{proof:thm:factor-1d:const:1} gives us
    \begin{equation}\label{eq:dfn:rv-1d:tail:4}
      \prob \Biggr(
      \bigcup_{\substack{I,I'\in\mathcal{D}_{\leq n}\\I\neq I'}} O_{I,I'}
      \cup \bigcup_{I\in\mathcal{D}_{\leq n}} D_{I}
      \Biggl)
      \leq \sum_{\substack{I,I'\in\mathcal{D}_{\leq n}\\I\neq I'}}
      \prob(O_{I,I'})
      + \sum_{I\in\mathcal{D}_{\leq n}} \prob(D_{I})
      \leq \frac{2^{3(n+2)}\Gamma^2}{2^{m_0/2}\eta_0^2}
      < 1.
    \end{equation}
    By~\eqref{eq:dfn:rv-1d:tail:4} and the definition of the events $O_{I,I'}$, $D_{I}$
    (see~\eqref{eq:dfn:rv-1d:tail:2}), we can find at least one $\theta\in\{\pm 1\}^{\mathcal{D}}$
    such that
    \begin{subequations}\label{proof:thm:factor-1d:almost-diag}
      \begin{align}
        |\langle T b_I^{(\theta)}, b_{I'}^{(\theta)}\rangle|
        & \leq \eta_0,
        &I,I'\in\mathcal{D}_{\leq n},\ I\neq I',&
                                                  \label{proof:thm:factor-1d:almost-diag:a}\\
        \Bigl|\langle T b_{I}^{(\theta)},
        b_{I}^{(\theta)}\rangle -
        \sum_{K\in\mathcal{B}_I}
        \langle Th_{K}, h_{K}\rangle\Bigr|
        &\leq \eta_0,
        &I\in\mathcal{D}_{\leq n}.&
                                    \label{proof:thm:factor-1d:almost-diag:b}
      \end{align}
    \end{subequations}
    Using~\eqref{eq:proof:thm:factor-1d:large}, \textrefp[C]{enu:c1}, \textrefp[C]{enu:c3} and that
    $\kappa=1$ by~\eqref{eq:proof:thm:factor-1d:coll:2}, we obtain
    \begin{equation*}
      \sum_{K\in\mathcal{B}_I}
      \langle Th_{K}, h_{K}\rangle
      \geq \sum_{K\in\mathcal{B}_I}
      \delta |K|
      = \delta |B_I|
      = \delta |I|,
      \qquad I\in\mathcal{D}_{\leq n}.
    \end{equation*}
    Combining this estimate with~\eqref{proof:thm:factor-1d:almost-diag:b} yields
    \begin{equation}\label{proof:thm:factor-1d:almost-diag:b:final:1}
      \langle T b_I^{(\theta)}, b_I^{(\theta)}\rangle
      \geq \delta |I| - \eta_0,
      \qquad I\in\mathcal{D}_{\leq n}.
    \end{equation}
    By Lemma~\ref{lem:block:basic-estimate-1d}, we have $\|b_I^{(\theta)}\|_2^2 = |I|$, thus,
    estimate~\eqref{proof:thm:factor-1d:almost-diag:b:final:1} implies
    \begin{equation}\label{proof:thm:factor-1d:almost-diag:b:final:2}
      \langle T b_I^{(\theta)}, b_I^{(\theta)}\rangle
      \geq (\delta - \eta_0 2^{n}) \|b_I^{(\theta)}\|_2^2,
      \qquad I\in\mathcal{D}_{\leq n}.
    \end{equation}

    Together, the estimates~\eqref{proof:thm:factor-1d:almost-diag:a}
    and~\eqref{proof:thm:factor-1d:almost-diag:b:final:2} give
    us~\eqref{proof:thm:factor-1d:overview:1}, that is
    \begin{subequations}\label{proof:thm:factor-1d:almost-diag:final}
      \begin{align}
        |\langle T b_I^{(\theta)}, b_{I'}^{(\theta)}\rangle|
        &\leq \eta_0,
        &I,I'\in\mathcal{D}_{\leq n},\ I\neq I',&
                                                  \label{proof:thm:factor-1d:almost-diag:final:a}\\
        \langle T b_I^{(\theta)}, b_I^{(\theta)}\rangle
        &\geq (\delta - 2^{n}\eta_0) \|b_I^{(\theta)}\|_2^2,
        &I\in\mathcal{D}_{\leq n}.&
                                    \label{proof:thm:factor-1d:almost-diag:final:b}
      \end{align}
    \end{subequations}
  \end{proofstep}

  \begin{proofstep}[Step~\theproofstep: Conclusion of the proof]
    By Theorem~\ref{thm:projection-1d}, the operators $B^{(\theta)} : W_n\to W_N$ and
    $A^{(\theta)} : W_N\to W_n$ given by
    \begin{subequations}\label{eq:proof:thm:factor-1d:operators}
      \begin{align}
        B^{(\theta)} f
        &= \sum_{I\in \mathcal{D}_{\leq n}} \frac{\langle f, h_I\rangle}{\|h_I\|_2^2}
          b_I^{(\theta)},
        && f\in W_n,
           \label{eq:proof:thm:factor-1d:operators:a}\\
        A^{(\theta)} f
        &= \sum_{I\in \mathcal{D}_{\leq n}}
          \frac{\langle f, b_I^{(\theta)}\rangle}{\|b_I^{(\theta)}\|_2^2} h_I,
        && f\in W_N
           \label{eq:proof:thm:factor-1d:operators:b}
      \end{align}
    \end{subequations}
    satisfy the estimates
    \begin{equation}\label{eq:proof:thm:factor-1d:estimates}
      \|B^{(\theta)}\|
      \leq 1
      \qquad\text{and}\qquad
      \|A^{(\theta)}\|
      \leq 1.
    \end{equation}
    The operator $P^{(\theta)} : W_N\to W_N$ defined by $P^{(\theta)} = B^{(\theta)}A^{(\theta)}$ is
    a norm~$1$ projection given by
    \begin{equation}\label{eq:proof:thm:factor-1d:P:definition}
      P^{(\theta)} f
      = \sum_{I\in\mathcal{D}_{\leq n}}
      \frac{\langle f, b_I^{(\theta)}\rangle}{\|b_I^{(\theta)}\|_2^2}
      b_I^{(\theta)},
      \qquad f\in W_N.
    \end{equation}
    Now, we define the subspace $Y$ by $Y = P^{(\theta)}(W_N)$, and note that the following diagram
    is commutative:
    \begin{equation}\label{eq:proof:thm:factor-1d:commutative-diagram:preimage}
      \vcxymatrix{W_n \ar[rr]^{\Id_{W_n}} \ar[d]_{B^{(\theta)}}
        && W_n\\
        Y \ar[rr]_{\Id_Y} && Y \ar[u]_{A^{(\theta)}_{|Y}}}
      \qquad \|B^{(\theta)}\|,\|A^{(\theta)}_{|Y}\| \leq 1.
    \end{equation}

    Next, we define $U^{(\theta)} : W_N\to Y$ by putting
    \begin{equation}\label{eq:proof:thm:factor-1d:almost-inverse}
      U^{(\theta)} f = \sum_{I\in\mathcal{D}_{\leq n}}
      \frac{\langle f, b_I^{(\theta)}\rangle}
      {\langle Tb_I^{(\theta)}, b_I^{(\theta)}\rangle}
      b_I^{(\theta)},
      \qquad f\in W_N.
    \end{equation}
    By the $1$-unconditionality of the Haar system, the definition of $P^{(\theta)}$
    (see~\eqref{eq:proof:thm:factor-1d:P:definition}) and the
    estimates~\eqref{proof:thm:factor-1d:almost-diag:final},
    \eqref{eq:proof:thm:factor-1d:estimates}, we obtain
    \begin{equation}\label{eq:proof:thm:factor-1d:U-bound}
      \|U^{(\theta)}\|
      \leq \frac{\|P^{(\theta)}\|}{\delta - \eta_0 2^{n}}
      \leq \frac{1}{\delta - \eta_0 2^{n}}.
    \end{equation}
    Moreover, for all $g = \sum_{I\in\mathcal{D}_{\leq n}} a_I b_I^{(\theta)} \in Y$, we have the
    following identity:
    \begin{equation}\label{eq:proof:thm:factor-1d:crucial-identity}
      U^{(\theta)}Tg - g
      = \sum_{\substack{I,I'\in\mathcal{D}_{\leq n}\\I \neq I'}} a_{I'}
      \frac{\langle T b_{I'}^{(\theta)}, b_I^{(\theta)}\rangle}
      {\langle Tb_I^{(\theta)}, b_I^{(\theta)}\rangle} b_I^{(\theta)}.
    \end{equation}
    Note that Lemma~\ref{lem:block:basic-estimate-1d} yields
    $|a_{I'}| \leq \frac{\|g\|_{W_N}}{\|b_{I'}^{(\theta)}\|_{W_N}}$, thus we obtain
    from~\eqref{proof:thm:factor-1d:almost-diag:final} that
    \begin{equation}\label{eq:proof:thm:factor-1d:crucial-estimate:2}
      \|U^{(\theta)}Tg - g\|_{W_N}
      \leq \biggl\| \sum_{\substack{I,I'\in\mathcal{D}_{\leq n}\\I' \neq I}} a_{I'}
      \frac{\langle T b_{I'}^{(\theta)}, b_I^{(\theta)}\rangle}
      {\langle Tb_I^{(\theta)}, b_I^{(\theta)}\rangle}
      b_I^{(\theta)}
      \biggr\|_{W_N}
      \leq \frac{\eta_0 2^{3(n+1)}}{\delta-\eta_0 2^{n}} \|g\|_{W_N}.
    \end{equation}
    Now, let $I : Y\to W_N$ denote the operator given by $Iy = y$.
    By~\eqref{proof:thm:factor-1d:const:1}, we have that
    $\frac{\eta_0 2^{3(n+1)}}{\delta-\eta_0 2^{n}} < 1$;
    hence~\eqref{eq:proof:thm:factor-1d:crucial-estimate:2} yields
    \begin{equation}\label{eq:proof:thm:factor-1d:crucial-estimate:3}
      \|(U^{(\theta)}TI)^{-1} g\|_{W_N}
      \leq \frac{1}{1-\frac{\eta_0 2^{3(n+1)}}{\delta-\eta_0 2^{n}}} \|g\|_{W_N}.
    \end{equation}
    By~\eqref{eq:proof:thm:factor-1d:U-bound}, \eqref{eq:proof:thm:factor-1d:crucial-estimate:3}
    and~\eqref{proof:thm:factor-1d:const:1}, the operator $V^{(\theta)} : W_N\to Y$ given by
    $V^{(\theta)}=(U^{(\theta)}TI)^{-1}U^{(\theta)}$ satisfies the estimate
    \begin{equation*}
      \|V^{(\theta)}\|
      \leq \frac{1}{\delta - \eta_0 (2^{n} + 2^{3(n+1)})}
      \leq \frac{1+\eta}{\delta},
    \end{equation*}
    and the following diagram is commutative:
    \begin{equation}\label{eq:proof:thm:factor-1d:commutative-diagram:image}
      \vcxymatrix{
        Y \ar@/^/[rr]^{\Id_Y} \ar@/_{10pt}/[dd]_I \ar[rd]_{U^{(\theta)}TI} & & Y\\
        & Y \ar[ru]^{(U^{(\theta)}TI)^{-1}} &\\
        W_N \ar[rr]_T & & W_N \ar[lu]_{U^{(\theta)}}
        \ar[uu]_{V^{(\theta)}}
      }
      \qquad \|I\|\|V^{(\theta)}\| \leq \frac{1+\eta}{\delta}.
    \end{equation}
    Merging the diagrams~\eqref{eq:proof:thm:factor-1d:commutative-diagram:preimage}
    and~\eqref{eq:proof:thm:factor-1d:commutative-diagram:image} yields
    \begin{equation}\label{eq:proof:thm:factor-1d:commutative-diagram:merged}
      \vcxymatrix{%
        W_n \ar@/_{25pt}/[ddd]_{E} \ar[rr]^{I_{W_n}} \ar[d]^{B^{(\theta)}} & & W_n\\
        Y \ar@/^/[rr]^{\Id_Y} \ar@/_{10pt}/[dd]_I \ar[rd]_{U^{(\theta)}TI} & & Y \ar[u]^{A^{(\theta)}_{|Y}}\\
        & Y \ar[ru]^{(U^{(\theta)}TI)^{-1}} &\\
        W_N \ar[rr]_T & & W_N \ar[lu]_{U^{(\theta)}}
        \ar[uu]_{V^{(\theta)}} \ar@/_{25pt}/[uuu]_{F}
      }
      \qquad \|E\| \|F\| \leq \frac{1+\eta}{\delta}.
    \end{equation}
  \end{proofstep}
  Finally, reviewing the construction of the block basis $b_I^{(\theta)}$,
  $I\in\mathcal{D}_{\leq n}$ (see~\eqref{eq:proof:thm:factor-1d:coll:1}
  and~\eqref{eq:proof:thm:factor-1d:coll:2}) and the definitions of the operators involved in
  diagram~\eqref{eq:proof:thm:factor-1d:commutative-diagram:merged}, $N$ must be at least $m_0+n$;
  hence, considering the constants defined in~\eqref{proof:thm:factor-1d:const:1}
  makes~\eqref{eq:thm:factor-1d:dim} an appropriate choice for $N$.
\end{myproof}

\subsection*{Acknowledgments}\hfill\\
\noindent
It is my pleasure to thank P.~F.~X.~Müller for many helpful discussions.  Supported by the Austrian
Science Foundation (FWF) Pr.Nr. P28352.

\bibliographystyle{abbrv}
\bibliography{bibliography}

\end{document}